\documentclass[a4paper,12pt]{article}

\usepackage{amsfonts}
\usepackage{amssymb}
\usepackage{latexsym}
\usepackage{amsthm}
\usepackage[dvips]{graphicx}

\usepackage{fancyhdr}

\usepackage[all]{xy}

\begin{document}

\theoremstyle{plain}       \newtheorem*{mainthm}{Theorem 1}
\theoremstyle{plain}       \newtheorem*{cor1}{Corollary 1}
\theoremstyle{plain}       \newtheorem*{cor2}{Corollary 2}
\theoremstyle{plain}       \newtheorem*{thm2}{Theorem 2}
\theoremstyle{plain}       \newtheorem{thm}{Theorem}[section]
\theoremstyle{definition}  \newtheorem{defn}[thm]{Definition}
\theoremstyle{plain}       \newtheorem{lemma}[thm]{Lemma}
\theoremstyle{plain}       \newtheorem{prop}[thm]{Proposition}
\theoremstyle{remark}   \newtheorem{rmk}[thm]{Remark}

\title{2-Frame Flow Dynamics and Hyperbolic Rank Rigidity in Nonpositive Curvature}
\author{David Constantine\footnote{Supported by NSF Graduate Research Fellowship}   
\footnote{\textsc{Department of Mathematics, University of Michigan, Ann Arbor, MI 48103 
U.S.A.}                  \textit{email:} \texttt{constand@umich.edu}}}
\date{\today}
\maketitle

%%%%%%%%%%%%%%%%%%%%%%

\begin{abstract}
This paper presents hyperbolic rank rigidity results for rank 1, nonpositively curved spaces.  Let $M$ be a compact, rank 1 manifold with nonpositive sectional curvature and suppose that along every geodesic in $M$ there is a parallel vector field making curvature $-a^2$ with the geodesic direction.  We prove that $M$ has constant curvature equal to $-a^2$ if $M$ is odd dimensional, or if $M$ is even dimensional and has sectional curvature pinched as follows: $-\Lambda^2 < K < -\lambda^2$ where $\lambda/\Lambda > .93$.  When $-a^2$ is the upper curvature bound this gives a shorter proof of the hyperbolic rank rigidity theorem of Hamenst\"{a}dt, subject to the pinching condition in even dimension; in all other cases it is a new result.  We also present a rigidity result using only an assumption on maximal Lyapunov exponents in direct analogy with work done by Connell.  The proof of the main theorem is simplified considerably by assuming strict negative curvature; in fact, in all dimensions but 7 and 8 it is then an immediate consequence of ergodicity of the $(dim(M)-1)$-frame flow.  In these exceptional dimensions, recourse to the dynamics of the 2-frame flow must be made and the scheme of proof developed there can be generalized to deal with rank 1, nonpositively curved spaces.
\end{abstract}

%%%%%%%%%%%%%%%%%%%%%%%

\section{Introduction}
Rank rigidity was first proved in the higher Euclidean rank setting by Ballman \cite{Ballman} and, using different methods, by Burns and Spatzier \cite{Burns-Spatzier}.  A manifold is said to have higher Euclidean rank if a parallel, normal Jacobi field can be found along every geodesic.  Ballman and Burns-Spatzier proved that if an irreducible, compact, nonpositively curved manifold has higher Euclidean rank, then it is locally symmetric.  Ballman's proof works for finite volume as well and the most general version of this theorem is due to Eberlein and Heber, who prove it under only a dynamical condition on the isometry group of $M$'s universal cover \cite{Eberlein-Heber}.  Hamenst\"{a}dt showed that a compact manifold with curvature bounded above by -1 is locally symmetric if along every geodesic there is a Jacobi field making curvature -1 with the geodesic direction \cite{Hamenstadt}.  She called this situation higher hyperbolic rank.  Shankar, Spatzier and Wilking extended rank rigidity into positive curvature by defining spherical rank.  A manifold with curvature bounded above by 1 is said to have higher spherical rank if every geodesic has a conjugate point at $\pi$, or equivalently, a parallel vector field making curvature 1 with the geodesic direction.  They proved that a complete manifold with higher spherical rank is a compact, rank one locally symmetric space \cite{sphericalrank}.

These results settle many rank rigidity questions, but leave questions about other curvature settings open (see \cite{sphericalrank} for an excellent overview).  In this paper we prove the following theorem, which can be applied to various settings in nonpositive curvature.

\begin{mainthm} Let $M$ be a compact, (Euclidean) rank 1, nonpositively curved manifold.  Suppose that along every geodesic in $M$ there exists a parallel vector field making sectional curvature 
$-a^2$ with the geodesic direction.  If $M$ is odd dimensional, or if $M$ is even dimensional and satisfies the sectional curvature pinching condition $-\Lambda^2 < K < -\lambda^2$ with $\lambda/\Lambda > .93$ then $M$ has constant negative curvature equal to $-a^2$.
\end{mainthm}

In strict negative curvature, the proof of this theorem is considerably simpler.  In fact, for negative curvature the full frame flow is ergodic under the conditions of Theorem 1 in all dimensions but 7 and 8 (\cite{Brin-Gromov} for odd dimensions, \cite{Brin-Karcher} for even dimensions).  Then the proof of Theorem 1 is immediate by considering a frame with dense orbit.  For dimensions 7 and 8 ergodicity of the $(n-1)$-frame flow holds under very strong curvature pinching (see \cite{Pollicott-Burns}) but under the curvature restrictions of Theorem 1 one only has ergodicity of the 2-frame flow.  Note that ergodicity of this flow alone does not establish Theorem 1 since the set of 2-frames giving the distinguished sectional curvature $-a^2$ may, a priori, have zero measure.  However, the ergodic theory of these types of flows, developed by Brin, proceeds via explicit geometric descriptions of the ergodic components, and this allows Theorem 1 to be obtained from the 2-frame flow dynamics alone.  The proof proceeding via 2-frame flow gives the result in the exceptional dimensions 7 and 8 in negative curvature.  In addition, it suggests an adaptation to the rank 1 nonpositive curvature setting, where the ergodic theory of frame flows has not been developed.  The simplifications possible in the strictly negative curvature setting will be noted throughout the paper, but here we observe that although obtaining this paper's result for nonpositively curved rank 1 spaces necessitates a more technical proof, the resulting theorem forms a better complement to the rank rigidity theorem of Ballmann and Burns-Spatzier.

Note that, unlike previous rank rigidity results, Theorem 1 allows for situations where the distinguished curvature $-a^2$ is not extremal.  However, in the cases where $-a^2$ is extremal the hypotheses of our theorem can be weakened, as demonstrated in section \ref{sec:parallel} of this paper.  The following two results are then easy corollaries of Theorem 1:

\begin{cor1}
Let $M$ be a compact, rank 1 manifold with sectional curvature $-1 \leq K \leq 0$.  Suppose that along every geodesic in $M$ there exists a Jacobi field making sectional curvature $-1$ with the geodesic direction.  If $M$ is odd dimensional, or if $M$ is even dimensional and satisfies the sectional curvature pinching condition $-1 \leq K < -.93^2$ then $M$ is hyperbolic.
\end{cor1}

\begin{cor2} \emph{(compare with Hamenst\"{a}dt \cite{Hamenstadt})}
Let $M$ be a compact manifold with sectional curvature bounded above by $-1$.  Suppose that along every geodesic in $M$ there exists a Jacobi field making sectional curvature $-1$ with the geodesic direction.  If $M$ is odd dimensional, or if $M$ is even dimensional and satisfies the sectional curvature pinching condition $-(1/.93)^2 < K \leq -1$ then $M$ is hyperbolic.
\end{cor2}

In Corollary 1 we obtain a new rank rigidity result analogous to those described above.  This is the first positive result for lower rank, i.e. when the distinguished curvature value is the lower curvature bound; \cite{sphericalrank} provides a discussion of counterexamples to lower spherical and Euclidean rank rigidity.  In Corollary 2 we obtain a shorter proof of Hamenst\"{a}dt's result, under an added pinching constraint in even dimension.

In \cite{Connell}, Connell showed that rank rigidity results can be obtained using only a dynamical assumption on the geodesic flow, namely an assumption on the Lyapunov exponents at a full measure set of unit tangent vectors.  His paper deals with the upper rank situations treated by Ballman, Burns-Spatzier and Hamenst\"{a}dt.  He proves that having the minimal Lyapunov exponent allowed by the curvature restrictions attained at a full measure set of unit tangent vectors is sufficient to apply the results of Ballman and Burns-Spatzier or Hamenst\"{a}dt.  In the lower rank setting of this paper, this viewpoint translates into

\begin{thm2}
Let $M$ be a compact, rank 1 manifold with sectional curvature $-a^2 \leq K \leq 0$, where $a > 0$.  Endow $T^1M$ with a fully supported ergodic measure; one can take the measure of maximal entropy or, if the curvature is known to be negative, the Liouville measure.  Suppose that for a full measure set of unit tangent vectors $v$ on $M$ the maximal Lyapunov exponent at $v$ is $a$, the maximum allowed by the curvature restriction.  If $M$ is odd dimensional, or if $M$ is even dimensional and satisfies the sectional curvature pinching condition $-a^2 \leq K < -\lambda^2$ with $\lambda/a > .93$ then $M$ is of constant curvature $-a^2$.
\end{thm2}

\noindent The adaptation of Connell's arguments for this setting is discussed in section \ref{sec:Lyap}.

The proof of Theorem 1 relies on dynamical properties of the geodesic and frame flows on nonpositively curved manifolds.  The work of Brin and others is the starting point for the proof; the results needed are summarized in Section \ref{sec:background} (see also \cite{Brin-survey} for a survey of the area).  Although none of his work is undertaken for rank 1 nonpositively curved manifolds, the ideas used in this paper to deal with that situation are clearly inspired by Brin's work.  The proof will proceed as follows.  We utilize the transitivity group $H_v$, defined for any vector $v$ in the unit tangent bundle of $M$, which acts on $v^{\perp} \subset T^1M$.  Essentially, elements of $H_v$ correspond to parallel translations around ideal polygons in $M$'s universal cover.  In negative curvature, Brin shows that this group is the structure group for the ergodic components of the frame flow (see e.g \cite{Brin-survey} or \cite{Brin-gpext}).  For the rank 1 nonpositive curvature case the definition of this group must be adjusted and we use only that it is the structure group for a subbundle of the frame bundle.  The considerations for the rank 1 case are discussed in section \ref{sec:rank1}. In section \ref{sec:trans} we show that $H_v$ preserves the parallel fields that make curvature $-a^2$ with the geodesic defined by $v$.  Finally, we apply results of Brin-Gromov (adapted to the rank 1 case) and Brin-Karcher on the 2-frame flow which imply that $H_v$ acts transitively on $v^\perp$ and conclude that the curvature of $M$ is constant.

I would like to thank Chris Connell for discussions helpful with the arguments in section \ref{sec:parallel} of this paper, Jeffrey Rauch for the proof of Lemma \ref{Rauchlemma}, and Ben Schmidt for helpful comments on this paper.  In particular, special thanks are due to my advisor, Ralf Spatzier, for suggesting this problem, for help with several pieces of the argument and for helpful comments on this paper.

%%%%%%%%%%%%%%%%%%%%%%

\section{Notation and background} \label{sec:background}

\subsection{Notation}
Let us begin by fixing the following notation:  

\begin{itemize}
\item $M$: a compact, rank 1 Riemannian manifold with nonpositive sectional curvature, $\tilde{M}$ its universal cover, $\tilde{M}(\infty)$ the boundary at infinity.
\item $T^1M$ and $T^1\tilde{M}$: the unit tangent bundles to $M$ and $\tilde{M}$, respectively.
\item $St_kM$: the $k$-frame bundle of ordered, orthonormal $k$-frames on $M$.
\item $g_t$: the geodesic flow on $T^1M$ or $T^1\tilde{M}$.
\item $F_t$: the frame flow on $St_kM$; when clear, $k$ will not be referenced.
\item $W^s_g$ and $W^u_g$: the foliations of $T^1\tilde{M}$ given by inward and outward pointing normal vectors to horospheres.
\item $\mu$: the Bowen-Margulis measure of maximal entropy on $T^1M$.
\item $\gamma_v (t)$: the geodesic in $M$ or $\tilde{M}$ with velocity $v$ at time 0.
\item $w_v (t)$: a parallel normal vector field along $\gamma_v (t)$ making the distinguished curvature $-a^2$ with $\dot{\gamma}_v (t)$.
\item $K(\cdot, \cdot)$: the sectional curvature operator.
\end{itemize}

Note that $\pi:St_kM \to T^1M$ mapping a frame to its first vector is a fiber bundle with structure group $SO(k-1)$ acting on the right; $St_nM$ is a principal bundle.  The measure $\mu$ is used in place of the standard Liouville measure as it has better known dynamical properties for rank 1, nonpositively curved spaces (see Section \ref{sec:rank1}).  In the negative curvature setting Liouville measure can be used.  Unless otherwise specified, $\mu$ and its product with the standard measure on the fibers of $St_kM$ inherited from the Haar measure on $SO(n-1)$ will be the measures used in all that follows.  In negative curvature, $W^s_g$ and $W^u_g$ are the stable and unstable foliations for the geodesic flow.

\subsection{Background}

In negative curvature, Brin develops the ergodic theory of frame flows as summarized below.  In Section \ref{sec:rank1} we will discuss how suitable portions of this setup can be generalized to the rank 1 setting.

First, the frame flow also gives rise to stable and unstable foliations $W^s_F$ and $W^u_F$ of $St_kM$ as shown by Brin (see \cite{Brin-survey} Prop. 3.2).  Brin notes that the existence of these foliations can be established in two ways, either by applying the work of Brin and Pesin on partially hyperbolic systems or by utilizing the exponential approach of asymptotic geodesics.  In the second approach the leaves of the foliation are constructed explicitly - they sit above the stable/unstable leaves for the geodesic flow, and $\alpha$ and $\alpha'$ are in the same leaf if the distance between $F_t(\alpha)$ and $F_t(\alpha')$ goes to zero as $t \to \infty$ for the stable leaves, or $t \to -\infty$ for the unstable leaves.  Here we present a Proposition that makes possible this definition of $W^*_F(\alpha)$ by establishing that frames asymptotic to $\alpha$ exist and are unique.  The proof follows the sketch given by Brin in \cite{Brin-survey}.

\begin{prop} \label{prop:leaf}
Let $v$ be a unit tangent vector and let $\alpha$ be a $k$-frame with first vector $v$.  Let $v' \in W^s_g(v)$ (respectively $W^u_g(v)$) so that the distance between $g_t(v)$ and $g_t(v')$ goes to zero exponentially fast as $t \to \infty$ (resp. $t \to -\infty$).  Then there exists a unique $k$-frame $\alpha'$ with first vector $v'$ such that the distance between $F_t(\alpha)$ and $F_t(\alpha')$ goes to zero as $t \to \infty$ (resp. $t \to -\infty$). 
\end{prop}

Note that in a compact, negatively curved any two asymptotic vectors approach each other exponentially fast so this Proposition allows us to define all leaves of the foliation.  In rank 1 spaces this may no longer be the case; thus we have added exponential approach as a hypothesis to the Proposition as it will be used for the rank 1 case later in the paper.

\begin{proof}
Assume $v' \in W^s_g(v)$; the unstable case is analogous.  Uniqueness of the limit is simple since it is clear that two different frames cannot both approach $\alpha$.  We thus have only to show existence.

For $t$ large enough, $g_t(v)$ and $g_t(v')$ are very close to each other, and then for every frame $\beta$ with first vector $g_t(v)$ there exists a unique frame, call it $f(\beta)$, which minimizes the distance from $\beta$ among frames with first vector $g_t(v')$.  Thus to approximate the unique frame $\alpha'$ we are looking for, consider the frames $\alpha'_t=F_{-t}(f(F_t(\alpha)))$.  We want to show that the $\alpha'_t$ have a limit as $t \to \infty$; this limit will clearly be our $\alpha'$.

Consider the sequence $\alpha'_n$ for $n \in \mathbb{N}$.  Since the frame flow is smooth, by choosing large enough $T$ the difference between the frame flow along the segments $[g_T(v), g_{T+1}(v)]$ and $[g_T(v'), g_{T+1}(v')]$ can be made arbitrarily small, and thus if the $\alpha'_n$ have a limit, it must be a  limit for the $\alpha'_t$.  Again, since the frame flow is smooth and since the distance between the geodesics decreases exponentially fast, the distance $d(\alpha'_n, \alpha'_{n+1})$ goes to zero exponentially fast as well.  Note then that $d(\alpha'_n, \alpha'_m)\leq \sum_{i=n}^{m-1}d(\alpha'_i, \alpha'_{i+1}) \leq \sum_{i=n}^{\infty}d(\alpha'_i, \alpha'_{i+1})$.  As the summands go to zero exponentially fast the last sum shown converges and given $\epsilon >0$ one can pick $n$ so large that this tail sum is less than $\epsilon$.  Then we see that the $\alpha'_n$ form a Cauchy sequence so they have a limit as desired.
\end{proof}

Let $p(v, v')$ be the map from the fiber of $St_kM$ over $v$ to the fiber over $v'$ that takes each $\alpha$ to $\alpha'=\pi^{-1}(v')(\alpha') \cap W^s_F(\alpha)$.  Note that $p(v, v')$ corresponds to a unique isometry between $v^{\perp}$ and $v'^{\perp}$ and commutes with the right action of $SO(k-1)$.  For most of this paper we will consider the maps $p(v, v')$ acting on 2-frames.  One can think of $p(v, v')(\alpha)$ as the result of parallel transporting $\alpha$ along $\gamma_v (t)$ out to the boundary at infinity of $\tilde{M}$ and then back to $v'$ along $\gamma_{v'} (t)$.  If $v'$ and $v$ belong to the same leaf of $W^u_g$ there is similarly an isometry corresponding to parallel translation to the boundary at infinity along $\gamma_{-v}$ and back along $\gamma_{-v'}$.  We will also denote this isometry by $p(v, v')$.  In the spirit of Brin (see \cite{Brin-survey} Defn. 4.4) we define the transitivity group at $v$ as follows:

\begin{defn}
Given any sequence $s=\{v_0, v_1, \ldots , v_k\}$ with $v_0=v, v _k=g_T(v)$ such that each pair $\{v_i, v_{i+1}\}$ lies on the same leaf of $W^s_g$ or $W^u_g$ we have an isomorphism of $v^{\perp}$ given by 
\[I(s) = F_{-T} \circ \prod_{i=0}^{k-1} p(v_i, v_{i+1}).\]
The closure of the group generated by all such isometries is denoted by $H_v$ and is called the transitivity group.
\end{defn}

\noindent The idea of the transitivity group is that it is generated by isometries coming from parallel translation around ideal polygons in $\tilde{M}$ with an even number of sides, such as the one shown in figure \ref{fig:rectangle}.

\begin{figure}
%\centering
\includegraphics[height=.45\textheight, width=.8\textwidth]{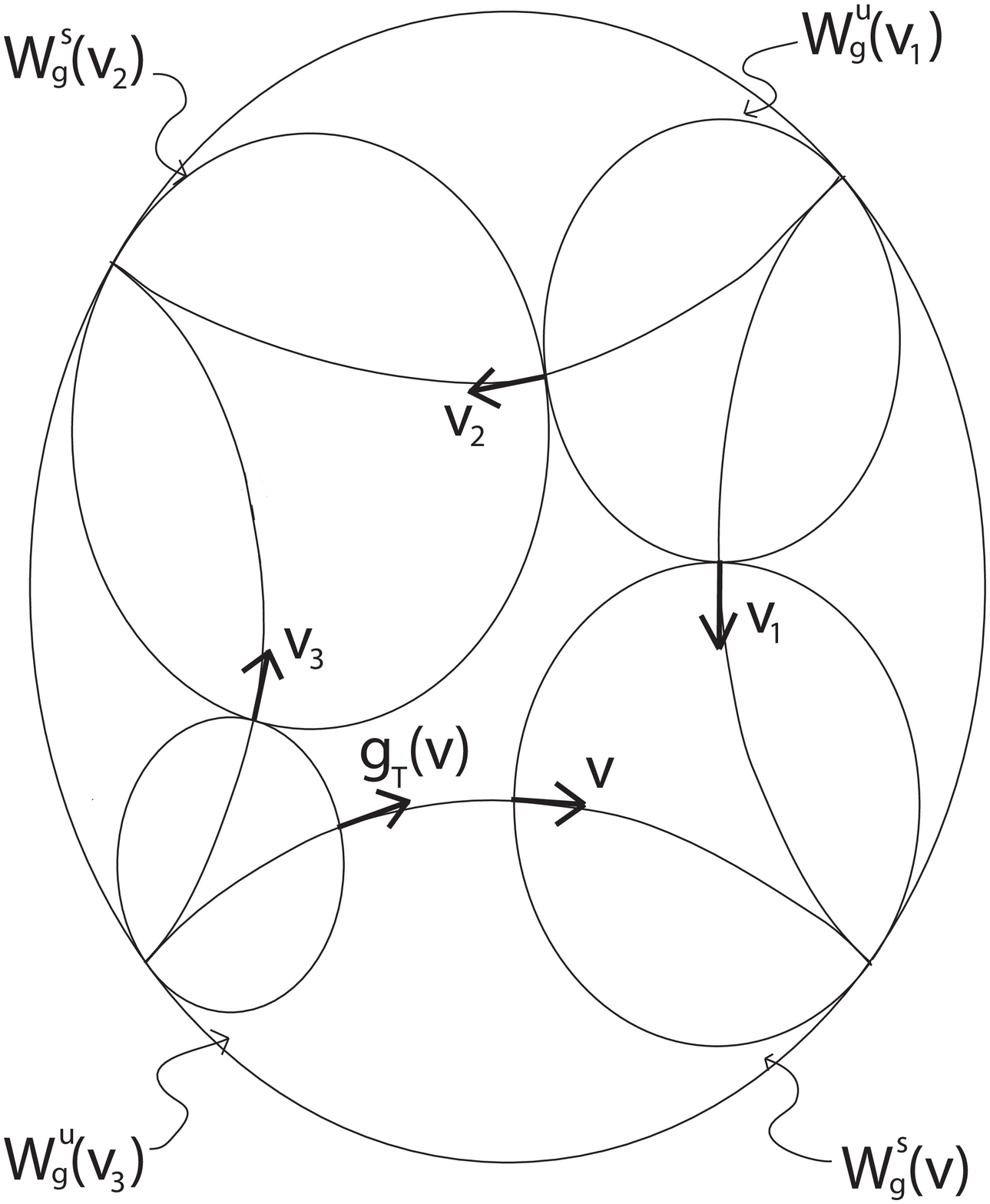}
\caption {An element of $H_v$}\label{fig:rectangle}
\end{figure}

Note that this definition differs slightly from that in Brin's work.  Brin requires that $v_k=v$ and thus there is no $F_{-T}$ term in his formula for $I(s).$  Brin proves that his group describes the ergodic components of the frame flow.  He shows in \cite{Brin-gpext} that the ergodic components are subbundles of $St_kM$ with structure group a closed subgroup of $SO(n-1)$, now acting from the left (see also \cite{Brin-survey} section 5 for an overview).  In addition, his proof demonstrates that the structure group for the ergodic component is the transitivity group (see \cite{Brin-survey} Remark 2 or \cite{Brin-gpext} Proposition 2).  Note that the action of $H_v$ can be taken to be a left action as it commutes with the $SO(k-1)$ action of the structure group.  This can be seen from noting that $p(v, v')(\alpha)\cdot g=p(v, v')(\alpha \cdot g)$ for any $g$ in the structure group, and that these maps define the transitivity group.  The proof that this group gives the ergodic components follows the Hopf argument for ergodicity, showing that the ergodic component is preserved under motion along stable and unstable leaves, and using the Birkhoff ergodic theorem to show that switching from stable to unstable also preserves the component.  

The transitivity group as defined here is certainly at least as large as that defined by Brin.  On the other hand, the addition of the $F_{-T}$ term here certainly preserves the ergodic components so this group still describes ergodic components and therefore is, in the end, the same as Brin's.  The advantage to this formulation of the definition is that it allows all ideal polygons, not just those that are `equilateral' in the sense that they can be traversed only by following leaves of the foliations.  The explicit geometric description of the ergodic components given here is the central inspiration for our proof.

We use two results on the ergodicity of the 2-frame flow in our proof.

\begin{thm} \label{thm:BG}
\emph{(Brin-Gromov \cite{Brin-Gromov} Proposition 4.3)} If $M$ has negative sectional curvature and odd dimension then the 2-frame flow is ergodic.
\end{thm}

\begin{thm} \label{thm:BK}
\emph{(Brin-Karcher \cite{Brin-Karcher})} If $M$ has sectional curvature satisfying $-\Lambda^2 < K < -\lambda^2$ with $\lambda / \Lambda > .93$ then the 2-frame flow is ergodic.
\end{thm}

Theorem \ref{thm:BK} is not directly stated as above in \cite{Brin-Karcher}, rather it follows from remarks made in section 2 of that paper together with Proposition 2.9 and the extensive estimates carried out in the later sections.

\subsection{A dynamical lemma}

The following dynamical lemma is one of the main tools in our proof.  It will be used in the proof of Lemma \ref{lemma:existence} and in the arguments of Section \ref{sec:trans}.

\begin{lemma} \label{recurrencelemma}
Suppose $\gamma(t)$ is a recurrent geodesic in $M$ with a parallel normal field $P(t)$ along it such that $K(P(t), \dot{\gamma}(t)) \rightarrow C$ as $t \rightarrow \infty$.  Then $K(P(t), \dot{\gamma}(t)) \equiv C$ for all $t$.
\end{lemma}

\begin{proof}
Since $\gamma(t)$ is recurrent we can take an increasing sequence $\{ t_k\} $ tending to infinity such that $\dot{\gamma}(t_k)$ approaches $\dot{\gamma}(0)$.  Since the parallel field $P(t)$ has constant norm and the set of vectors in $\dot{\gamma}_v(t)^{\perp}$ with this norm is compact, we can, by passing to a subsequence, assume that $P(t_k)$ has a limit $G(0)$.  Extend $G(0)$ to a parallel vector field $G(t)$ along $\gamma(t)$.

By construction, $K(G(0), \dot{\gamma}(0)) = \lim_{k \rightarrow \infty} K(P(t_k), \dot{\gamma}(t_k)) = C$.  In addition, for any real number $T$, the recurrence $\dot{\gamma}(t_k) \rightarrow \dot{\gamma}(0)$ implies recurrence $\dot{\gamma}(t_k+T) \rightarrow \dot{\gamma}(T)$.  By continuity of the frame flow, we get that $P(t_k+T) \rightarrow G(T)$ for the vector field $G$ defined above.  Thus $G(t)$ makes curvature $C$ with $\dot{\gamma}(t)$ for any time $t$.

We can repeat the same argument as above, letting $G(t)$ recur along the same sequence of times to produce $G_1(t)$, and likewise $G_i(t)$ recur to produce $G_{i+1}(t)$, forming a sequence of fields all making curvature identically $C$ with the geodesic direction.  Now, observe that $G(0) = P(0)\cdot g$ for some $g\in SO(n-1)$.  Note here that $g$ is not well-defined by looking at $P$ and $G$ alone, but will be well defined if we consider $n$-frame orbits with second vector $P$ recurring to $n$-frames with second vector $G(0)$; this is the $g$ we utilize.  By construction and the fact that the $SO(n-1)$ 
action commutes with parallel translation, $G_i(0)=P(0)\cdot g^{i+1}$.  $SO(n-1)$ is compact, so the $\{g^i\}$ have convergent subsequences.  In addition, since the terms of this sequence are all iterates of a single element, we can, by adjusting terms of such a subsequence by suitable negative powers of $g$, have the subsequence converge to the identity.  Choose a subsequence $\{i_j\}$ such that $g^{i_j+1} \rightarrow id$ as $j \rightarrow \infty$.  These $G_{i_j}(t)$ approach our original field $P(t)$ showing that $P$ makes constant curvature $C$ with $\dot{\gamma}$ as well.
\end{proof}

%%%%%%%%%%%%%%%%%%%%%%%%%%%%%%%%%%%

\section{Extensions to rank 1 spaces} \label{sec:rank1}

In this section we discuss some details of the extension to rank 1, nonpositively curved spaces, and note 
how the necessary results on the dynamics of the frame flow can be appropriated to this situation.

\subsection{The measure of maximal entropy}
First we discuss the measure of maximal entropy $\mu$.  This measure was developed for rank 1 spaces by Knieper \cite{Knieper} and is constructed there as follows.  Let $\{\nu_p\}_{p\in\tilde M}$ be the Patterson-Sullivan measures on $\tilde M(\infty)$.  Fix any $p\in \tilde M$.  Let $\mathcal{G}^E$ be the set of pairs $(\xi, \eta)$ in $\tilde M(\infty)$ that can be connected by a geodesic.  Then $d\bar\mu (\xi, \eta) = f(\xi, \eta) d\nu_p(\xi)d\nu_p(\eta)$ defines a measure on $\mathcal{G}^E$; $f$ is a positive function which can be chosen to make the measure invariant under $\pi_1(M)$.

Let $P:T^1\tilde M \to \mathcal{G}^E$ be the projection $P(v) = (\gamma_v(-\infty), \gamma_v(\infty))$.  We then get a $g_t$ and $\pi_1$ invariant measure $\mu$ on $T^1\tilde M$ by setting, for any Borel set $A$ of $T^1\tilde M$
$$\mu(A) = \int_{\mathcal{G}^E} vol(\pi(P^{-1}(\xi, \eta) \cap A))d\bar\mu (\xi, \eta),$$
where here $\pi:T^1M\to M$ is the standard projection and $vol$ is the volume element on the submanifolds $P^{-1}(\xi, \eta)$.  

We need three key facts about this measure.  First, $\mu$ is ergodic for the geodesic flow (see \cite{Knieper} Theorem 4.4).  Second, $\mu$ has full support.  This follows from the facts that $\mu$ is supported on the rank 1 vectors (see \cite{Knieper} again) and that the rank 1 vectors are dense in $T^1M$ (see e.g. \cite{Ballmann-axial}).  Third, $\mu$ is absolutely continuous for the foliations $W^s_g$ and $W^u_g$.  Absolute continuity of a measure for a foliation is a way of asking that a Fubini-like property hold for the foliation when integrating with respect to the measure.  In our situation, we can justify Fubini's theorem for this measure and the foliations $W^s_g$ and $W^u_g$ directly.  Up to multiplication by the positive function $f$, $\mu$ is locally a product measure.  Variation of the set $A$ in the $W^s_g$ and $W^u_g$ directions is measured by the $d\nu_p$ measures as well as the part of $vol$ that measures variation normal to the geodesic flow direction when $P^{-1}(\xi, \eta)$ is a flat of dimension greater than one.  The rest of the measure $vol$ measures variation of the set $A$ in the geodesic flow direction.  Combining these measures together as a product measure with the scaling by $f$ yields $\mu$, demonstrating that Fubini-style arguments for this measure hold.

\subsection{The transitivity group}
Next, we want to extend the definition of the transitivity group to rank 1, nonpositively curved spaces spaces.  The central difficulty here is that the distance between asymptotic geodesics may approach a nonzero constant.  Thus it is no longer clear whether foliations like $W^s_F$ and $W^u_F$ can be defined.  We deal with this difficulty by avoiding defining foliations for the frame flow, but still defining maps $p(v, v')$ used to produce a transitivity group.  In Section \ref{sec:trans} we will show that the transitivity group preserves the distinguished parallel fields.  The technical points involved in the definitions of the $p$ maps and the transitivity group are necessary to make that proof work.

First, let $\Omega' \subset \tilde{M}(\infty)$ be the set of endpoints $\xi$ of recurrent, rank 1 geodesics such that almost all geodesics ending at $\xi$ are recurrent.  Since recurrence and rank 1 are full measure conditions for the ergodic measure $\mu$ and $\mu$ is absolutely continuous for $g_t$ this is a full measure set of $\tilde{M}$ for any Patterson-Sullivan measure on $\tilde{M}$.  In particular, it is dense as the Patterson-Sullivan measures have full support.

The first two conditions are placed on $\Omega'$ to allow the proof of the following Lemma; the third will be needed for our work in Section \ref{sec:trans}.

\begin{lemma} \label{lemma:strictlyasymptotic}  There exists a full measure subset $\Omega \subset \Omega'$ such that if $\xi \in \Omega$ then any two geodesics $\gamma_1$ and $\gamma_2$ with $\gamma_1(\infty)=\gamma_2(\infty)=\xi$ are exponentially strictly asymptotic, that is, the distance between them goes to zero exponentially fast as they head towards $\xi$.
\end{lemma}

\begin{proof}
First we show that the distance between $\gamma_1$ and $\gamma_2$ goes to zero (see also \cite{Knieper} Prop. 4.1).  As $\xi \in \Omega'$ it is the end point of a rank 1 recurrent geodesic, call this geodesic $\gamma_v$.  Suppose $\gamma_1$ and $\gamma_2$ are not strictly asymptotic.  Then $\gamma_v$ is not strictly asymptotic to one of these geodesics, without loss of generality, say $\gamma_1$.  Since $\gamma_v$ is recurrent there exists a sequence $\{\phi_i\}$ of isometries of $\tilde M$ and a sequence of real numbers $\{t_i\}$ tending to infinity such that $\phi_i(g_{t_i}(v)) \to v$ as $i \to \infty$.  Consider the sequences of geodesics $\{\phi_i(\gamma_v)\}$ and $\{\phi_i(\gamma_1)\}$.  By choice of the $\phi_i$ the first sequence converges to $\gamma_v$.  Also, since $\gamma_1$ is asymptotic to $\gamma_v$, after perhaps passing to a subsequence, the second sequence converges to a geodesic, call it $\bar\gamma$.  As $\gamma_v$ and $\gamma_1$ are not \emph{strictly} asymptotic, $\bar\gamma \neq \gamma$, but since they are asymptotic, $\bar \gamma(-\infty)=\gamma_v(-\infty)$ and $\bar\gamma(\infty)=\gamma_v(\infty)$.  Then the flat strip theorem (see \cite{Eberlein-Oneill}) implies that $\gamma_v$ and $\bar\gamma$ bound a totally geodesically embedded flat strip, contradicting the fact that $\gamma_v$ is rank 1.

Now we show that this convergence is exponential for almost all $\xi$.  Since the manifold is rank 1 and the geodesic flow is ergodic for $\mu$, $\mu$ is a hyperbolic measure for $g_t$ (see the supplement on Pesin theory to \cite{H-K}).  Then Pesin theory tells us that there exists a geodesic $c$ whose central foliation consists solely of $\mathbb{R}\dot{c}(0)$.  Then every stable Jacobi field $J$ along $c$ eventually decreases in magnitude.  In fact, if $\{J_i\}$ is a basis of the stable Jacobi fields, we can pick $T\in \mathbb{R}$ large enough that for any $J_i$, $\frac{d}{dt} |J_i(t)| <a<0$ for some real number $a$ a positive fraction of the time $[0,T]$.  As this is true for the basis, it is true for any stable Jacobi field.
Then by continuity of the geodesic flow and its derivatives the Jacobi fields, we can chose a small $\epsilon >0$ and consider the set of all geodesics $d$ with $d(t)\in B_\epsilon(c(t))$ for $t\in [0,T]$ so that we have for all stable Jacobi fields $J'$ along $d$ that $\frac{d}{dt} |J'(t)| <a<0$ a positive fraction of the time $[0, T]$.

Consider pairs of asymptotic geodesics $d_1$ and $d_2$ in this set.  We claim that the distance between the geodesics decreases over the interval $[0, T]$ at least by a fraction bounded away from 0 which is independent of the choice of $d_1$ and $d_2$.  To see this, take $v_s$ to be the shortest path in $W^s_g(\dot{d}_1(0))$ with $v_0=\dot{d}_1(0)$ and $v_1=\dot{d}_2(0)$.  The geodesics $v_s(t)=exp(t v_s)$ give a variation of asymptotic geodesics connecting $d_1$ and $d_2$.  Now, at any $t\in [0, T]$ the distance between $d_1(t)$ and $d_2(t)$ along $W^s_g(\dot{d}_1(0))$ is 
\begin{equation}
D(t)=\int_0^1 \Big\|\frac{D}{ds} v_s(t)\Big\| ds. \label{eqn:D}
\end{equation}
Now $\frac{D}{ds} v_s(t)$ for a fixed $s_0$ is a stable Jacobi field along the geodesic $v_{s_0}(t)$.  By our choice of the $\epsilon$-neighborhood in which $v_s(t)$ lies, we know that we have 
\[\frac{d}{dt} \Big\|\frac{D}{ds} v_s(t)\Big\| <a \]
a positive fraction of the time, independently of our choice of $d_1$ and $d_2$.  Since this is true for all the Jacobi fields in the integrand for Equation \ref{eqn:D} we see that $D$ will decrease over the interval $[0,T]$ by at least by some positive fraction independent of the choice of $d_1$ and $d_2$.  If this is true for the distance $D$ then it is true for the actual distance in $M$ between $d_1(t)$ and $d_2(t)$ so this distance also decreases by a positive fraction.

Now, let $U$ be a small enough neighborhood of $\dot{c}(0)$ in $T^1M$ that all $v\in T^1M$ remain within $\epsilon/2$ of $c([0,T])$ for the next $T$ units of time.  By the Birkhoff ergodic theorem applied to the geodesic flow, there exists a full measure set of $v\in T^1M$ such that $g_t(v)$ returns to $U$ a positive fraction of the time.  Let $\Omega \subset \Omega'$ be the endpoints at infinity of the geodesics generated by such vectors.  We now show that this is the set we are looking for.  Let $d$ be a geodesic asymptotic to $\gamma_v$ for $v\in \Omega$.  We show that it approaches $\gamma_v$ exponentially, and hence any two geodesics asymptotic to $\gamma_v$ approach each other exponentially.  We know from the first part of this proof that the distance between $d$ and $\gamma_v$ goes to zero, so after some finite amount of time the distance between them is less than $\epsilon/2$.  Then for a positive fraction of the time after that we know $\dot \gamma_v$ is in $U$ and hence both $\gamma_v$ and $d$ will spend a positive fraction of the time in the $\epsilon$-neighborhood of $c([0,T])$ discussed above.  As shown in the previous paragraph, during each visit to this neighborhood of $c$ the distance between $d$ and $\gamma_v$ decreases at least by a set fraction.  As this happens a positive fraction of the time as $t$ goes to infinity, we see that $d$ approaches $\gamma_v$ exponentially fast, proving the Lemma.
\end{proof}

Let $v' \in W^s_g(v)$ (the case $v' \in W^u_g(v)$ proceeds in a similar manner).  We make definitions of 
$p(v, v')$ in two cases.  
\\

\noindent \textbf{Case I:} $\gamma_v(\infty) \in \Omega$.  Lemma \ref{lemma:strictlyasymptotic} tells us that we have exponential convergence of the geodesics in question.  Then Proposition \ref{prop:leaf} allows us to define $p(v, v')$ mapping frames over $v$ to frames over $v'$ as in the negative curvature case.  
\\

\noindent \textbf{Case II:} $\gamma_v(\infty) \notin \Omega$.  We define a family of maps $\{p_{\{\xi_n\}}(v, v')\}$ in the following manner.  Let $\gamma_v(\infty)=\xi$.  Consider all sequences of points $\{\xi_n\}$ in $\Omega$ that approach $\xi$ in the sphere topology on $\tilde M(\infty)$.  As noted, $\Omega$ is dense in $\tilde M(\infty)$, so we can find such sequences approaching any $\xi$.  Let $c_n$ and $c'_n$ be the geodesics connecting the footpoints of $v$ and $v'$ to $\xi_n$ such that $c_n(0)$ is the footpoint of $v$ and $\dot{c}'_n(0) \in W^s_g(\dot{c}_n(0))$ (see figure \ref{fig:defn}).  The maps $p(\dot{c}_n(0), \dot{c}'_n(0))$ are defined under Case I.  As $n$ tends to infinity, $\dot{c}_n(0) \to v$ and $\dot{c}'_n(0) \to v'$ so limit points of the maps $\{p(\dot{c}_n(0), \dot{c}'_n(0))\}$ will give maps from frames over $v$ to frames over $v'$.  Let us restrict the allowed sequences $\{\xi_n\}$ to only those for which $\{p(\dot{c}_n(0), \dot{c}'_n(0))\}$ has a unique limit; call that limit $p_{\{\xi_n\}}(v, v')$.  These will be the allowed maps for the second case of the definition.
\\

\begin{figure}
%\centering
\includegraphics[height=.35\textheight, width=.8\textwidth]{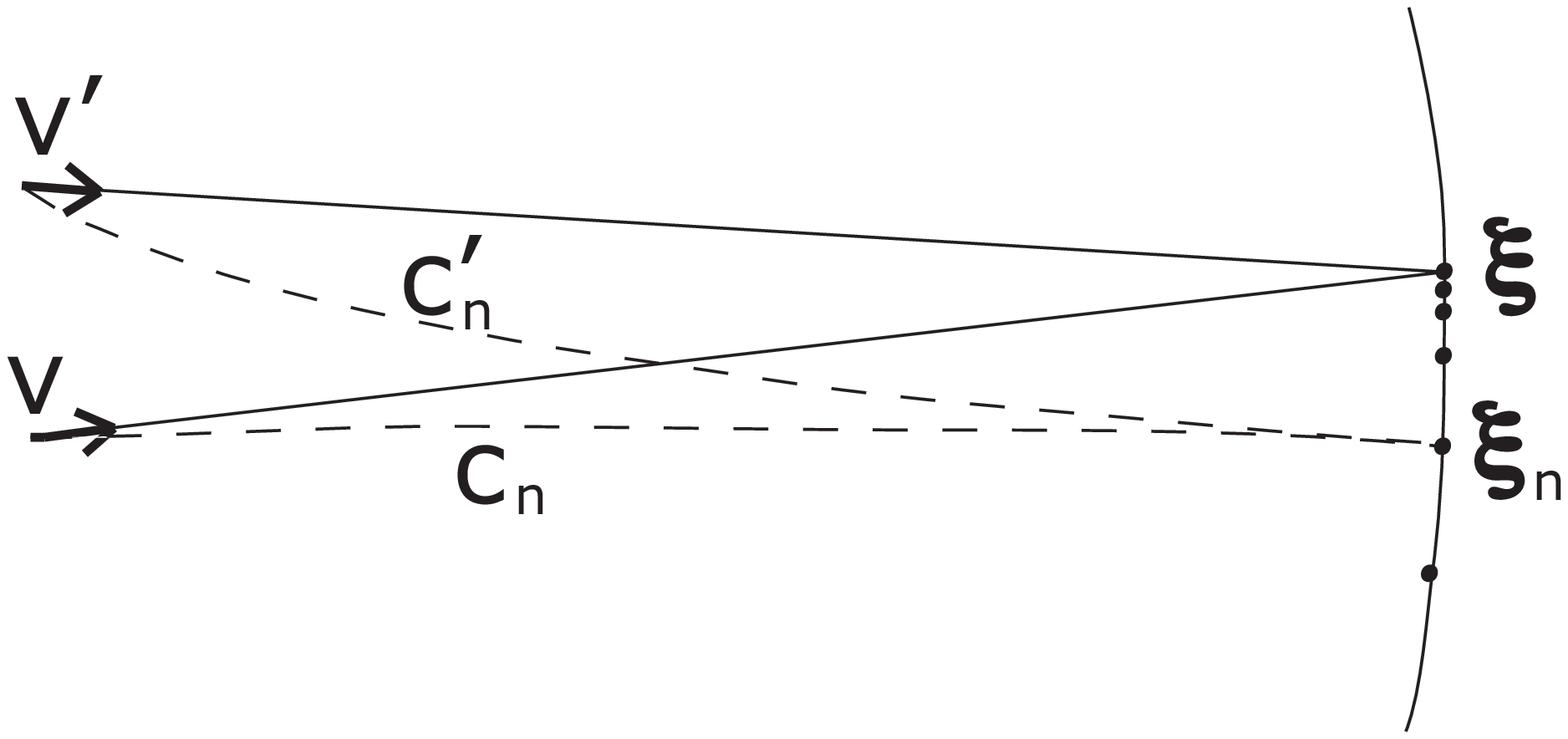}
\caption {adjusting the definition of the transitivity group}\label{fig:defn}
\end{figure}

As before, the transitivity group will be defined as a composition of the $p$ and $p_{\{\xi_n\}}$ maps corresponding to translations around ideal polygons.  Again, nonpositive curvature necessitates some technical considerations.  First, as noted in \cite{Eberlein} 1.11, not all pairs of points on $\tilde{M}(\infty)$ can be connected by geodesics.  However, we note in the following Lemma, that a given point can be connected to almost all other points on $\tilde{M}(\infty)$.

\begin{lemma} \label{lemma:connect} Let $\xi$ be an element of $\tilde{M}(\infty)$ and let $A_\xi$ be the subset of $\tilde{M}(\infty)$ consisting of points that can be connected to $\xi$ by geodesics.  Then $A_\xi$ contains an open, dense set.
\end{lemma}

\begin{proof}
This Lemma simply works out some consequences of \cite{Ballmann-axial}.  Let $A'_\xi \subset A_\xi$ be the set of all points at infinity which can be connected to $\xi$ by a geodesic that does not bound a flat half plane.  Ballmann's Theorem 2.2 (iii) tells us that $A'_\xi$ contains all endpoints of periodic geodesics that do not bound a flat half plane.  Together with his Theorem 2.13, this  implies that the set $A'_\xi$ is dense.  In addition, Ballmann's Lemma 2.1 implies that $A'_\xi$ is open, proving the Lemma.\end{proof}

We also introduce the following technical criterion:
\\

\noindent \textit{\textbf{Composition Criterion.}}  Let $v$ be in $T^1\tilde{M}$ and let $\xi=\gamma_v(\infty), \eta=\gamma_v(-\infty) \notin \Omega$ be the endpoints at infinity of $\gamma_v$.  Let $\xi_n\to \xi$ and $\eta_n\to \eta$ be sequences in $\Omega$.  We will say the pair $(\{\xi_n\}, \{\eta_n\})$ satisfies the \textit{composition criterion for $v$} if $\eta_n \in A_{\xi_n}$ for all $n$ and $\gamma_n \rightarrow \gamma_{v}$ as $n \rightarrow \infty$, where $\gamma_n$ is the geodesic connecting $\xi_n$ and $\eta_n$.
\\

This criterion will be required of pairs $(\{\xi_n\}, \{\eta_n\})$ if we are to compose the maps $p_{\{\xi_n\}}$ and $p_{\{\eta_n\}}$ in forming elements of the transitivity group.  It will be important in the proof of Proposition \ref{prop:case2}.  We must first, however, establish that, given $\{\xi_n\}$ and $v$, there exist sequences $\{\eta_n\}$ that satisfy the composition criterion for $v$ with $\{\xi_n\}$.  Without this fact the definition we will make of the transitivity group could be vacuous.

\begin{lemma} \label{lemma:existence}
Given $v \in T^1M$ with $\gamma_v(\infty)=\xi, \gamma_v(-\infty)=\eta$ and a sequence $\xi_n \rightarrow \xi$ in $\Omega$, there exists a sequence $\eta_n \rightarrow \eta$ in $\Omega$ such that $\{\xi_n\}$ and $\{\eta_n\}$ satisfy the composition criterion for $v$.
\end{lemma}

\begin{proof}
For any $\zeta \in \tilde{M}(\infty)$ let $pr_\zeta: \tilde{M} \to \tilde{M}(\infty)$ be the projection defined by setting $pr_\zeta(y)$ equal to $\gamma(\infty)$ where $\gamma$ is the geodesic with $\gamma(-\infty)=\zeta$ and $\gamma(0)=y$. Let $x$ be the footpoint of $v$ and $B_R(x)$ be the ball of radius $R$ around $x$ in $\tilde{M}$.  

Lemma 3.5 in \cite{Knieper} tells us that given $x$, there exists an $R>0$ such that $pr_{\xi_n}(B_R(x))$ contains an open set $U$ in $\tilde{M}(\infty)$.  Examining Knieper's proof one sees that $R$ can be taken to be any number greater than the distance from $x$ to some rank 1 geodesic $c$ which has an endpoint at $\xi_n$.  Consider the geodesic joining $x$ and $\xi_n$.  Since $\xi_n\in \Omega$, it is the endpoint of a rank 1, recurrent geodesic, call it $\gamma$.  Lemma \ref{lemma:strictlyasymptotic} then implies that the geodesic $c_n$ from $x$ to $\xi_n$ is strictly asymptotic to $\gamma$ and then Lemma \ref{recurrencelemma} can be applied with $C=0$ to show that $c_n$ must be rank 1 as $\gamma$ is.  Thus the rank 1 geodesic $c$ needed for Knieper can be taken to be $c_n$.  It is distance 0 from $x$, therefore $R$ can be taken to be arbitrarily small; in particular, take $R_n=1/2^n$ for $\xi_n$.

For each $R_n$, the open set $U_n$ given to us by Knieper's Lemma contains elements of the set $\Omega \cap A_{\xi_n}$ since $A_{\xi_n}$ is open and dense and $\Omega$ has full measure.  Pick $\eta_n \in U_n \cap \Omega \cap A_{\xi_n}$ to form our sequence $\{\eta_n\}$.  Then the geodesics $\gamma_n$ connecting $\xi_n$ and $\eta_n$ enter $B_{1/2^n}(x)$ for all $n$, and as $\xi_n \to \xi$ and $\eta_n \to \eta$ we must have $\gamma_n \to \gamma_{v}$.  Thus $(\{\xi_n\}, \{\eta_n\})$ satisfies the composition criterion for $v$ as desired.
\end{proof}

The transitivity group is defined via the following two definitions.  We start by defining its action on the frames above one particular vector $v$.

\begin{defn} \label{defn:rk1transgpv}
Let $v\in T^1M$ be such that $\gamma_v(\infty)$ and $\gamma_v(-\infty)$ are in $\Omega$.  Consider any sequence $s=\{v_0, v_1, \ldots , v_k\}$ with $v_0=v, v_k=g_T(v)$ for some real $T$, such that each pair $\{v_i, v_{i+1}\}$ lies on the same leaf of $W^s_g$ or $W^u_g$.  Furthermore, take for each pair $\{v_i, v_{i+1}\}$ falling under Case II a choice of a sequence $\{\xi^i_n\} \subset \Omega$ as described above.  We require that $(\{\xi^i_n\}, \{\xi^{i+1}_n\})$ satisfies the composition criterion for $v_{i+1}$.  Then we have an isomorphism of $v^{\perp}$ given by 
\[I(s) = F_{-T} \circ \prod_{i=0}^{k-1} p_-(v_i, v_{i+1}).\]
Here $p_-(v_i, v_{i+1})=p(v_i, v_{i+1})$ when $\{v_i, v_{i+1}\}$ falls under Case I and $p_-(v_i, v_{i+1})=p_{\{\xi^i_n\}}(v_i, v_{i+1})$ when $\{v_i, v_{i+1}\}$ falls under Case II.  The closure of the group generated by all such isometries is denoted by $H_v$.
\end{defn}

We extend the action of this group to any $w\in T^1M$ by connecting $v$ to $w$ by a segment of an ideal polygon.  To do so we simply need a point $\xi$ in $A_{\gamma_v(\infty)} \cap A_{\gamma_w(\infty)}$.  By Lemma \ref{lemma:connect} this set is open and dense, so in fact we can choose $\xi \in \Omega \cap A_{\gamma_v(\infty)} \cap A_{\gamma_w(\infty)}$.  Let $g$ be the isometry from $v^{\perp}$ to $w^{\perp}$ given by frame flow along the segment connecting $v$ and $w$ via $\xi$.  More specifically, let $v_1$ lie on the geodesic connecting $\gamma_v(\infty)$ and $\xi$ such that $v_1 \in W^s_g(v)$, let $v_2$ lie on the geodesic connecting $\xi$ and $\gamma_w(\infty)$ such that $v_2 \in W^u_g(v_1)$, and let $T\in \mathbb{R}$ be such that $g_T(w) \in W^s_g(v_2)$.  Then let

\begin{equation} \label{eqn:g}
g=F_{-T} \circ p_-(v_2, g_T(w)) \circ p_-(v_1, v_2) \circ p(v, v_1).
\end{equation}

In the negative curvature case, it is clear that $H_w = gH_vg^{-1}$.  Thus we complete the definition  of the transitivity group by making the following

\begin{defn} \label{defn:rk1transgp}
Let $H_w:=gH_vg^{-1}.$
\end{defn}

\begin{rmk}
Note that the choices of $v$ and $\xi$ only affect the group $H_w$ up to multiplication by an element of $H_v$, so the specific choices are not relevant.  In addition, attempting to define elements of $H_w$ for vectors $w$ for which neither endpoint is in $\Omega$ by ideal polygons based at $\gamma_w$ is problematic as, due to the composition criterion, the composition of such elements may not be in the group.  Hence we have define such $H_w$ via $H_v$ where no such issues arise. In the end we have a well defined action of an abstract group $H$ isomorphic to $H_v$ on the frame bundle, which in the negatively curved case essentially reduces to Brin's definition.  Again, as the $p_-(v, v')$ maps constructed here are invariant under elements of the structure group $SO(k-1)$, the action of $H$ commutes with the action of $SO(k-1)$ and thus takes the form of a left action.
\end{rmk}

\subsection{The subbundle given by $H$}

We now construct a subbundle of $St_kM$ for any $k \leq n$ with an action of $H$ on it.

\begin{defn} \label{defn:subbundle}  Given a $k$-frame $\alpha$ based above a vector $v\in \Omega$ let $Q(\alpha)\subseteq St_kM$ be the smallest set containing $\alpha$ and closed under all $h\in H_v$, $F_t$ for all $t$ and all isometries $g$ as in Equation \ref{eqn:g}.
\end{defn}  

\begin{prop} \label{prop:subbundle}
$Q(\alpha)$ is a subbundle of $St_kM$.
\end{prop}

\begin{proof}
Since for any $w\in T^1M$ we have found an isometry $g$ as in Definition \ref{defn:rk1transgp}, we see that $\pi(Q(\alpha))=T^1M$.

Let $\bar{\alpha}$ be an extension of the $k$-frame $\alpha$ to an $n$-frame with first $k$ vectors given by $\alpha$.  We first show that $Q(\bar\alpha)$ is a subbundle.  By construction, $Q(\bar\alpha)$ admits an action of $H$, an abstract group isomorphic to $H_v$.  It is clear that $Q(\bar\alpha) \cap \pi^{-1}(w)$ is the $H_w$ orbit of $g(\bar\alpha)$ for any $w\in T^1M$, where $g$ is as in Equation \ref{eqn:g}.   Furthermore, $H$ acts freely on $St_nM$ so all orbit types of this action are the same.  Thus Theorem 5.8 from \cite{Bredon} applies and $\pi: Q(\bar\alpha) \to T^1M$ is a fiber bundle with structure group $H$ as desired.

For $k<n$, embed $SO(n-k)$ into $SO(n-1)$, the structure group for $St_nM$ so that it acts on the last $n-k$ vectors in a given frame.  The map $\bar\pi: St_nM / SO(n-k) = St_kM \to T^1M$ is the subbundle of $k$-frames.  To produce $Q(\alpha)$ we would like to apply the same process to $Q(\bar\alpha)$ but must proceed carefully.  Let 
\[K_{\bar\alpha}=\{\kappa \in SO(n-k) | \bar\alpha \cdot \kappa = h(\kappa) \cdot \bar\alpha \mbox{ for some } h(\kappa)\in H_v\}\] 
where $SO(n-k)$ acts on $\bar\alpha$ via the same embedding.  This is the stabilizer of the first $k$ vectors in $\bar\alpha$ (that is, $\alpha$) in the subgroup of the structure group that preserves the $H_v$-orbit of $\bar\alpha$.  We now examine this stabilizer for any other frame $\bar\alpha'$ in $Q(\bar\alpha)$.  Any such $\bar\alpha'$ takes the form $h'\cdot g \cdot \bar\alpha$ for some $h' \in H_w$ and $g$ as in Eqn. \ref{eqn:g}.  We then compute
\begin{displaymath}\begin{array}{rl}
K_{\bar\alpha'}= & \{\kappa\in SO(n-k) | \bar\alpha' \cdot \kappa = h(\kappa) \cdot \bar\alpha' \mbox{ for some } h(\kappa)\in H_w\}\\
                            = & \{\kappa\in SO(n-k) | h'\cdot g \cdot \bar\alpha \cdot \kappa = h(\kappa) h'\cdot g \cdot \bar\alpha \mbox{ for some } h(\kappa)\in H_w\}.\\
\end{array} \end{displaymath}
But $h'\cdot g \cdot \bar\alpha \cdot \kappa = h(\kappa) h'\cdot g \cdot \bar\alpha$ if and only if $\bar\alpha \cdot \kappa = g^{-1} \cdot (h')^{-1} h(\kappa) h'\cdot g \cdot \bar\alpha$, and $g^{-1} \cdot (h')^{-1} h(\kappa) h'\cdot g$ is an element of $H_v$ so we see that $K_{\bar\alpha'}=K_{\bar\alpha}$ for all $\bar\alpha' \in Q(\bar\alpha)$.  We can refer to this group simply as $K$, and we note that $K \hookrightarrow H$ by $\kappa \mapsto h(\kappa)^{-1}$.  Thus, we can obtain $\pi: Q(\alpha) \to T^1M$ as $\bar\pi: Q(\bar\alpha)/K \to T^1M$.  The fibers of this map are of the form $H/K$ everywhere so again we can apply \cite{Bredon} to see that we have a fibration, as desired.
\end{proof}

\begin{rmk}
We have proved that $Q(\alpha)$ is a topological sub-fiber bundle - nothing has been claimed about smoothness.  $C^1$-smoothness of $Q(\alpha)$ in the negative curvature case is proven by Brin and is key to his proof that $Q(\alpha)$ is the ergodic component containing $\alpha$; here, however, we need only the topological result to appropriate the needed results from Brin-Gromov.  
\end{rmk}

\begin{prop}
The transitivity group $H$ acts transitively on the fiber of 2-frames over any $v\in T^1M$.
\end{prop}

\begin{proof}
First, note that when $n$ is even we are restricted to strict negative curvature and this result is Theorem \ref{thm:BK} due to Brin and Karcher.  When $n$ is odd the result is Theorem \ref{thm:BG} due to Brin and Gromov and found in section 4 of \cite{Brin-Gromov}.  They discuss the proof only in the strict negative curvature case, but it works perfectly well in nonpositive curvature.  We have included it here for completeness.

Our work in Proposition \ref{prop:subbundle} produced a subfibration $\pi: Q(\alpha) \to T^1M$ with fiber $H/K$ of the fibration $\pi: St_kM \to T^1M$ with fiber $SO(n-1)/SO(n-k)$.  Let us now restrict our attention to 2-frames, and specifically to $St_2M|_p$, the restriction of the 2-frame bundle to those frames based at a point $p$ of $M$.  We get bundles

\begin{displaymath}
\xymatrix{
H/K \ar[r] & Q(\alpha)|_p/K \ar[dd]^{\pi} & &  S^{n-2} \ar[r] & St_2M|_p \ar[dd]^{\pi} \\
                 &                                                  & \ar@{^{(}->}[r]^{i}  & &  \\
                 & S^{n-1}                                    & &                        & S^{n-1} }
\end{displaymath}
where $S^{n-2}=SO(n-1)/SO(n-2)$ and $S^{n-1}$ is the unit tangent sphere above $p$.  Take $b_0\in  S^{n-1}$ and $x_0 \in \pi^{-1}(b_0) \subset H/K$.  These fibrations, together with the inclusion map $i$ give the following commutative diagram, which connects the homotopy long exact sequences for the fibrations by the induced inclusion map $i_*$ (see \cite{Hatcher} Theorem 4.41):

\begin{displaymath}
\xymatrix{
\pi_{n-1}(Q(\alpha)|_p/K, x_0) \ar@{^{(}->}[d]^{i_*} \ar[r]^{\pi_*} & \pi_{n-1}(S^{n-1}, b_0) \ar[d]^{\cong} \ar[r]^{\bar\partial} & \pi_{n-2}(H/K, x_0)\ar@{^{(}->}[d]^{i_*}  \\
\pi_{n-1}(St_2M|_p, x_0) \ar[r]^{\pi_*} & \pi_{n-1}(S^{n-1}, b_0) \ar[r]^{\partial} & \pi_{n-2}(S^{n-2}, x_0) }
\end{displaymath}
Note that $\partial = i_* \circ \bar\partial$.  Now suppose $H$ does not act transitively on the fiber of two frames over some $v \in T^1M$.  Then $H/K \subsetneq S^{n-2}$ so $i_*=0$ on $\pi_{n-2}(H/K, x_0)$ and thus $\partial = 0$ on $\pi_{n-1}(S^{n-1}, b_0)$.  This implies that the map $\pi$ admits a section, thus giving a nonvanishing vector field on $S^{n-1}$. If $n$ is odd this is a contradiction.
\end{proof}

%%%%%%%%%%%%%%%%%%%%%%%%%%%%%%%%%%%

\section{The transitivity group and distinguished vector fields} \label{sec:trans}

As noted in the Introduction, the transitivity group is crucial to this paper's proof.  In this section we 
show that certain distinguished vector fields $w_v(t)$ along $\gamma_v(t)$ are preserved under the action of the transitivity group and use this result to prove Theorem 1.  Throughout we will utilize our dynamical lemma, Lemma \ref{recurrencelemma} with $C=-a^2$.  Consider, for example, the ideal rectangle defined by $v$, $v_1$, $v_2$ and $v_3$ as pictured in figure \ref{fig:rectangle}.  If $\gamma_{v_1}$ and $\gamma_{v_3}$ are positively recurrent and If $\gamma_v$ and $\gamma_{v_2}$ are negatively recurrent, Lemma \ref{recurrencelemma} implies that the element of $H$ corresponding to this ideal polygon preserves the distinguished fields.  The following arguments show how this idea can be worked out for \emph{all} ideal polygons, first in the negative curvature case and then in the general case.

\subsection{The negative curvature case}

In the negative curvature case the argument of this section is considerably simpler, so we discuss it first.  Consider the situation depicted in figure \ref{fig:rectangle}.  Lemma \ref{recurrencelemma} shows that, when $\gamma_{v_1}$ is recurrent in forward time, the map $p(v, v_1)$ preserves the distinguished vector fields in the sense that it sends a vector from one such field, $w_v(0)$, to a vector from another such field along $\gamma_{v_1}$.  Thus, if in figure 1 we have that $\gamma_{v_1}$ and $\gamma_{v_3}$ are recurrent in positive time and $\gamma_v$ and $\gamma_{v_2}$ are recurrent in negative time, then the element of $H_v$ given by parallel translation around this ideal rectangle will map $w_v(0)$ to another element of $v^{\perp}$ which is in a parallel field along $\gamma_v$ making curvature $-a^2$.  If these sort of recurrence properties held for all `equilateral' ideal polygons based at $v$ we would have that the transitivity group preserves the distinguished vector fields.  We cannot assure that these recurrence properties are always present, but ergodicity of the geodesic flow on $M$ indicates that they will be present almost all the time.  Furthermore, the fact that we have defined elements of the transitivity group using the continuous foliations provided by Brin can be used to argue that elements of the transitivity group depend continuously on the choice of ideal polygon.  Thus the transitivity group will preserve the distinguished vector fields.

\subsection{The general case}

For the general case we need arguments to deal with the problem of pairs of geodesics that are asymptotic but not strictly asymptotic, and the fact that we know longer know we have a continuous foliation.  By assumption, the frame flow preserves the distinguished vector fields.  The only question in terms of how they behave under the action of elements from the transitivity group is how they behave when they are transferred across corners of the ideal polygons.  As Lemma \ref{recurrencelemma} shows, when the geodesics involved are strictly asymptotic and the second geodesic is recurrent, the fields are transferred as desired.  Thus, there are two problems to deal with: when the second geodesic is not recurrent, and when the geodesics are not strictly asymptotic.  The new definition of the transitivity group provides a way to deal with both of these issues.  

First, note that under the Case II of Definition \ref{defn:rk1transgp}, we have defined the maps $p_{\{\xi_n\}}(v,v')$ as limits of the maps from the first case of the definition.  Thus, to show that a distinguished field $w_\gamma$ is preserved by some $p_{\{\xi_n\}}(v,v')$ we need to realize $w_\gamma$ as a limit of distinguished fields along the geodesics $c_n$ used to define the map $p_{\{\xi_n\}}(v,v')$ (see Figure \ref{fig:defn}).  In view of this fact we will work with distinguished fields that can be realized as limits, and ensure that this property of arising as a limit is also preserved by the $p_-(v,v')$ maps.  In particular, we will consider distinguished fields $w_v$ that arise as limits of distinguished fields $w_{c_n}$ along geodesics $c_n$ as depicted in Figure \ref{fig:corner} and show that such a field is transferred by a map $p_-(v,v')$ to a field $w_{v'}$ arising as a limit of fields $w_{d_n}$ along geodesics $d_n$ which connect $\gamma_{v'}(0)$ to $\{\eta_n\} \to \gamma_{v'}(-\infty)$.  If the next corner to be traversed falls under Case II, the sequence $\{\eta_n\}$ is determined by the map $p_{\{\eta_n\}}(v', v'')$; otherwise we are free to take any sequence.  The arguments are slightly different in the two cases so we address them separately:
\\

\begin{figure}
\centering
\includegraphics[height=.3\textheight, width=1\textwidth]{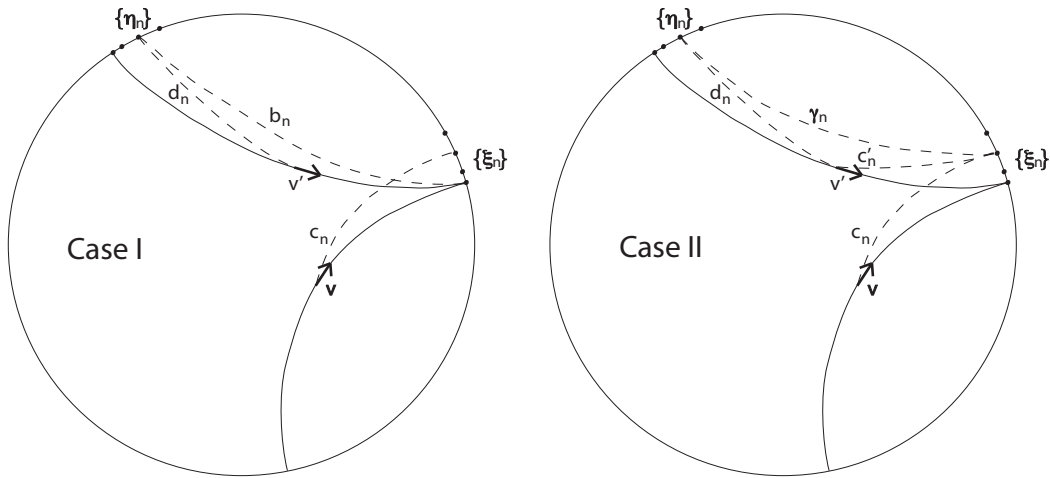}
\caption{geodesic configurations for Prop. 4.1 and 4.2} \label{fig:corner}
\end{figure}

\noindent \textbf{Case I.}  

\begin{prop}
Suppose $v\in T^1M$ falls under Case I, and that $w_v$ is the limit of $w_{c_n}$.  Then for any $v' \in W^s_g(v)$ (respectively $W^u_g(v)$), $w_{v'}:=p(v, v')(w_v)$ is a distinguished field along $\gamma_{v'}$ arising as a limit of distinguished fields $w_{d_n}$.
\end{prop}

\begin{proof}
Assume $v' \in W^s_g(v)$; the proof for the unstable case is essentially the same.  Since we are under Case I, $\gamma_{v'}$ is limit of recurrent geodesics $\gamma_{v'_i}$ for $v'_i \in W^s_g(v)$.  The maps $p(v, v'_i)$ preserve distinguished fields as demonstrated in Lemma \ref{recurrencelemma} and $p(v, v'_i) \to p(v, v')$, so $w_{v'}$ will be a distinguished field as well.

Now we need to demonstrate $w_{\gamma_{v'}}$ as a limit in the proper way.  Construct geodesics $d_n$ connecting $v'(0)$ and $\eta_n$ and $b_n$ connecting $\gamma_v(\infty)$ and $\eta_n$ as in Figure \ref{fig:corner}.  The field $w_{b_n}:=p(v, \dot{b}_n(0))(w_v)$ will be a distinguished field by the argument of the previous paragraph.  Likewise, the fields $w_{d_n}:=p(\dot{b}(0)_n, \dot{d}_n(0))(w_{b_n})$ will be distinguished fields as $\eta_n \in \Omega$ falls under Case I.  Since the $w_{b_n} \rightarrow w_{v'}$ it is clear that the $w_{d_n} \rightarrow w_{v'}$ and we are done.
\end{proof}

\noindent \textbf{Case II.} 

\begin{prop}\label{prop:case2}
Suppose $v\in T^1M$ falls under Case II, and that $w_v$ is the limit of $w_{c_n}$.  Then for any $v' \in W^s_g(v)$ (respectively $W^u_g(v)$), $w_{v'}:=p_{\{\xi_n\}}(v, v')(w_v)$ is a distinguished field along $\gamma_{v'}$ arising as a limit of distinguished fields $w_{d_n}$.
\end{prop}

\begin{proof}
Again, assume $v' \in W^s_g(v)$.  Since the maps $p(\dot{c}_n(0), \dot{c}'_n(0))$ are under Case I, they preserve distinguished fields.  Thus, $w_{v'}$, which is defined as the limit of $p(\dot{c}_n(0), \dot{c}'_n(0))(w_{c_n})$, will be a distinguished field.

Immediately, we have that $w_{v'}$ arises as a limit of distinguished fields along the geodesics $c'_n$. As in Figure \ref{fig:corner}, let $\gamma_n$ be the geodesic connecting $\xi_n$ and $\eta_n$ and let $d_n$ be the geodesic connecting the footpoint of $v'$ and $\eta_n$.  If $\gamma_{v'}(-\infty)$ is not in $\Omega$ the map $p_{\{\eta_n\}}(v', v'')$ supplies the sequence $\{\eta_n\}$.  In this case we have required that $(\{\xi_n\}, \{\eta_n\})$ satisfies the composition criterion for $v'$, so $c'_n$, $d_n$ and $\gamma_n$ all approach each other (and $\gamma_{v'}$) as $n \to \infty$.  If $\gamma_{v'}(-\infty)$ is in $\Omega$ it is easy to see that these geodesics still all converge as otherwise we would find a flat strip along $\gamma_{v'}$.  Using $p$ maps under Case I, the fields $w_{c'_n}$ can be transfered to distinguished fields $w_{\gamma_n}$ along $\gamma_n$ and subsequently to distinguished fields $w_{d_n}$ along $d_n$.  It is then clear that $w_{c'_n}, w_{\gamma_n}$ and $w_{d_n}$ all limit on $w_{v'}$; specifically, $w_{d_n} \to w_{v'}$ shows that $w_{v'}$ arises as a limit in the desired manner.
\end{proof}

This work proves

\begin{prop} \label{prop:preserves}
The transitivity group preserves distinguished vector fields that arise as limits of distinguished fields in the correct manner.
\end{prop}

\subsection{Proof of the main theorem}

We can now apply the results of Brin-Karcher and of Brin-Gromov as adapted to the rank 1 situation and prove Theorem 1 easily.

\begin{mainthm}
Let $M$ be a compact, rank 1, nonpositively curved manifold.  Suppose that along every geodesic in $M$ there exists a parallel vector field making sectional curvature $-a^2$ with the geodesic direction.  If $M$ is odd dimensional, or if $M$ is even dimensional and satisfies the sectional curvature pinching condition $-\Lambda^2 < K < -\lambda^2$ with $\lambda/\Lambda > .93$ then $M$ has constant negative curvature equal to $-a^2$.
\end{mainthm}

\begin{proof}
We showed in Proposition \ref{prop:preserves} that the sectional curvature $-a^2$ fields that arise in the desired way as limits are preserved by the transitivity group.  In the setting of the theorem, the adapted results of Brin-Gromov and Brin-Karcher tell us that the transitivity group acts transitively on $v^\perp \subset T^1M$.  In particular, by considering the orbit of a distinguished field that arises correctly as a limit we see that $K(\cdot, v)$ is identically $-a^2$, and the theorem is proved.
\end{proof}

%%%%%%%%%%%%%%%%%%%%%%%%%%%%

\section{Parallel fields and Jacobi fields} \label{sec:parallel}

In \cite{sphericalrank} a distinction is made between `weak' and `strong' rank.  The existence of \emph{parallel} fields making extremal curvature is called strong rank; the existence only of \emph{Jacobi} fields making extremal curvature is called weak rank.  A parallel field scaled by a solution to the real variable version of the Jacobi equation (where the standard derivative replaces the covariant derivative) produces a Jacobi field.  Thus, a proof under the less stringent hypothesis of weak rank implies a proof for strong rank.  Hamenst\"{a}dt's is the sole result prior to this paper for weak rank.  She states her main theorem for parallel fields only, but she shows in Lemma 2.1 that in negative curvature a Jacobi field making maximal curvature is a parallel field scaled by a function \cite{Hamenstadt}.  Essentially, she shows that Jacobi fields making maximal curvature grow at precisely the rate one finds for the constant curvature case.  Connell accomplishes the same in \cite{Connell} Lemma 2.3.  This, together with some of the arguments below, shows that these Jacobi fields are in fact parallel fields scaled by an appropriate function.  Therefore, Corollary 2 is a weak rank result, needing only the Jacobi field hypothesis.

In this section we show that Jacobi fields making \emph{minimal} curvature with the geodesic direction are also scaled parallel fields.  This will justify the phrasing of Corollary 1 as a weak rank result.  In this section, $\langle \cdot, \cdot \rangle$ will denote the Riemannian inner product and $R(\cdot, \cdot) \cdot$ the curvature tensor.

First, note that we need only consider non-vanishing Jacobi fields; hence it will be enough to prove that stable and unstable Jacobi fields are scaled parallel fields.  Stable Jacobi fields are those which have norm approaching zero as $t \to \infty$; unstable Jacobi fields have the same property in the negative time direction.  Suppose $J(t)$ is a stable Jacobi field along the geodesic $\gamma(t)$ making curvature $-a^2$ with the geodesic (take $a > 0$ now), where $-a^2$ is the curvature minimum for the manifold 
(the modifications of what follows for unstable Jacobi fields are straightforward).  The Rauch Comparison Theorem (see \cite{doC} Chapt 10, Theorem 2.3) can be used to show that

\begin{equation} \label{eq:geq}
|J(t)| \geq |J(0)| e^{-at}.
\end{equation}

We would like to show that equality is achieved in (\ref{eq:geq}).  Write $J(t)=j(t)U(t)$ where $j(t)=|J(t)|$ and $U(t)$ is a unit vector field.  Then the Jacobi equation for $J$ reads:

\begin{equation} \label{eq:jacobi}
j''U + 2j'U' + jU'' + jR(\dot{\gamma}, U)\dot{\gamma} = 0
\end{equation}
where $j'$ denotes the standard derivative and $U'$ denotes covariant derivative.  Taking the inner product of (\ref{eq:jacobi}) with $U$ and noting that $\langle U'', U\rangle = - \langle U', U' \rangle$ we obtain

\begin{equation} \label{eq:simpjacobi}
j'' - j(\langle U', U'\rangle +a^2) = 0.
\end{equation}

We now know the following about the magnitude of $J$: $j \geq 0$ by definition, $\lim_{t \to \infty} j(t) = 0$ since $J$ is a stable Jacobi field, and $j'' \geq a^2 j$ by (\ref{eq:simpjacobi}).  These allow the following conclusion; its proof was shown to the author by Jeffrey Rauch:

\begin{lemma} \label{Rauchlemma}
Let $j$ be a non-negative, real valued function satifsying $j'' \geq a^2 j$ and $\lim_{t \to \infty} j(t) = 0$.  Then $j(t) \leq j(0) e^{-at}$.
\end{lemma}

\begin{proof}
We have that $a^2j-j'' \leq 0.$  On the interval $0 \leq t \leq R$ for $R\gg1$ define $g_R$ by $g_R(0)=j(0)$, $g_R(R)=j(R)$ and $a^2g_R-g_R'' = 0$.  Note that as $R \to \infty$, $g_R \to j(0)e^{-at}$.  We claim that $j \leq g_R$; the Lemma follows in the limit.

This claim is essentially the maximum principle.  First, $j \leq g_R$ holds at $0$ and $R$.  Now suppose $j-g_R$ has a positive maximum at $c \in (0, R)$.  Then $(j''-g_R'')(c) \leq 0$.  However, we know $a^2(j-g_R) - (j''-g_R'') \leq 0$, so a positive value of $j-g_R$ at $c$ together with a negative value of $j''-g_R''$ yields a contradiction.  Therefore $j \leq g_R$ holds on all of $[0, R]$ as desired.
\end{proof}

This Lemma, together with equation (\ref{eq:geq}), tells us that $|J(t)| = |J(0)|e^{-at}$.  Examining equation (\ref{eq:simpjacobi}) we see that having the growth rate $e^{-at}$, as in the constant curvature $-a^2$ case, implies that $U' =0$, that is, $J$ is a scaled parallel field, as desired.

%%%%%%%%%%%%%%%%%%%%%%%%%%%%%%%%

\section{The dynamical perspective} \label{sec:Lyap}

In this section we discuss how the results of Connell in \cite{Connell} can be adapted to prove Theorem 2 as a simple consequence of Corollary 1.  The necessary changes are for the most part cosmetic; the discussion here is included for completeness, but the author does not claim to have added anything of substance to Connell's work.  The notation below that has not already been assigned follows Connell's for ease of reference.

Recall that Lyapunov exponents are a tool for measuring long-term asympotic growth rates in dynamical systems (see Katok and Mendoza's Supplement to \cite{H-K} section S.2 for an exposition).  In the setting of the geodesic flow they can be defined as follows.  Let $v \in T^1M$ and $u \in v^{\perp}$.  Let $J_u(t)$ be the unstable Jacobi field along $\gamma_v$ with initial condition $J_u(0) = u$.  Then, the \emph{positive Lyapunov exponent at $v$ in the $u$-direction} is
\[\lambda_v^+(u) = \limsup_{t \to \infty} \frac{1}{t} log|J_u(t)|. \]
Define \[\lambda_v^+ = \max_{u \in v^{\perp}} \lambda_v^+(u). \]
This is the maximal Lyapunov exponent at $v$; the curvature bound $-a^2 \leq K$ (again, take $a>0$) implies that $\lambda_v^+ \leq a$.  Let
\[ \Omega = \{v \in T^1M : \lambda_v^+ = a\}. \]
We can now rephrase Theorem 2 more succinctly.

\begin{thm2}
Let $M$ be a compact manifold with sectional curvature $-a^2 \leq K \leq 0$.  Suppose that $\Omega$ has full measure with respect to a geodesic flow-invariant measure $\mu$ with full support.  If $M$ is odd dimensional, or if $M$ is even dimensional and satisfies the sectional curvature pinching condition $-a^2 \leq K < -\lambda^2$ with $\lambda/a > .93$ then $M$ is of constant curvature $-a^2$.
\end{thm2}

Connell shows in the upper rank case that the dynamical assumption implies the geometric one, that is, that the manifold in fact has higher rank, allowing the application of an appropriate rank rigidity theorem.  He first shows (\cite{Connell} Proposition 2.4) that along a closed geodesic $\lambda_v^+ = a$ implies the existence of an unstable Jacobi field making curvature $-a^2$ with the geodesic direction.  Essentially, if the Jacobi field giving rise to the Lyapunov exponent does not have this curvature, it will 
continually see non-extremal curvature a positive fraction of the time as it moves around the closed geodesic.  This contradicts the supposed value of the Lyapunov exponent.  The lower curvature bound version of the argument is exactly the same as that presented by Connell, with the proper inequalities reversed; also note that the work in section \ref{sec:parallel} of this paper gives the results analogous to Connell's Lemma 2.3 necessary for the argument.

It is clear that if a dense set of geodesics have the distinguished Jacobi fields, then all geodesics will.  Since the velocity vectors for closed geodesics are dense in $T^1M$, Connell finishes his proof in section 3 of \cite{Connell} by showing that these vectors are all in $\Omega$ and using the argument of the previous paragraph.  Adapted to the setting of Theorem 2 the argument runs as follows.  If $w \in T^1M$ is tangent to a closed geodesic and $\lambda_w^+ < a$ the previous paragraph implies that any unstable Jacobi field along $\gamma_w$ must make curvature strictly greater than $-a^2$ a positive fraction of the time.  By continuity, this will also be true of any unstable Jacobi field along a geodesic $\gamma_v$ in a sufficiently small neighborhood of $\gamma_w$ (in the Sasaki metric on $T^1M$).  The ergodic theorem implies that for a full measure set of $v \in T^1M$, $\gamma_v$ will spend a positive fraction of its life in this small neighborhood of the periodic geodesic $\gamma_w$; the positivity follows from the fact that $\mu$ has full support.  The intersection of this full measure set with the full measure set $\Omega$ thus contains vectors $v$ which have $\lambda_v^+ = a$ but spend a positive fraction 
of their life so close to $\gamma_w$ that no Jacobi fields along them can make the minimal curvature $-a^2$ with the geodesic direction during this fraction of the time.  In fact, since $\gamma_w$ is compact, so is the closure of this small neighborhood and therefore the curvature between these Jacobi fields and the geodesics, when in this neighborhood, can be bounded away from $-a^2$, i.e. $K(J_u, \dot{\gamma}_v)>c>-a^2$ for a fixed $c$.  Having this curvature bound a positive fraction of the time contradicts $\lambda_v^+ = a$; therefore all closed geodesics must lie in $\Omega$ and the argument is complete.

Again, this version of the argument, relevant for the lower curvature bound situation, is the same as that presented by Connell with the proper inequalities reversed.  Thus, the dynamical assumption implies the geometric assumption of Corollary 1 and Theorem 2 follows.  Note that for these arguments the extremality of the distinguished curvature is essential and we do not obtain a result that parallels Theorem 1 in allowing non-extremal distinguished curvature.

%%%%%%%%%%%%%%%%%%%%%%%%%%%%%%%%%%%%%%%

\section{Conclusion} \label{sec:conclusion}

We conclude with a few remarks on possible extension of this work.  Note that in even dimension a result directly parallel to our odd dimensional result cannot be hoped for.  Since parallel translation preserves the complex structure on a K\"{a}hler manifold the 2-frame flow will not be ergodic (see \cite{Brin-Gromov} for some results on unitary frame bundles).  These known counterexamples to ergodic frame flow are excluded by requiring $-1<K<-1/4$, leading Brin to conjecture that strict 1/4-pinching implies that the frame flow is ergodic (\cite{Brin-survey} Conjecture 2.6).  A positive answer to this conjecture, or any extended results for ergodicity of the 2-frame flow in negative curvature would extend the results on rank rigidity presented here correspondingly, using the same proof as presented above.  One still hopes that lower hyperbolic rank rigidity (in the sense that higher rank implies the space is locally symmetric) could be true without any curvature pinching in even dimensions, and the result here as well as the extensive analogous results for the other rank rigidity theorems seem to make such a theorem more likely.  However, such a result would call for a significantly different method of proof from that presented here.

\end{document}